\def\version{01/06/2017 \quad version 5
\hfill\href{http://arxiv.org/abs/1405.3695}{arXiv:1405.3695}
}

\documentclass[11pt,a4paper]{amsart}
\usepackage{fullpage}

\usepackage{amssymb,amscd}

\usepackage[all]{xy}
\newdir{ >}{{}*!/-8pt/@{>}}

\usepackage{mathrsfs}

\usepackage[numeric,lite]{amsrefs}

\usepackage{hyperref}

\usepackage{pigpen}
\def\PO{\text{\pigpenfont R}}

\renewcommand{\thefootnote}{\fnsymbol{footnote}}
\long\def\symbolfootnote[#1]#2{\begingroup%
\def\thefootnote{\fnsymbol{footnote}}\footnote[#1]{#2}\endgroup}

\newtheorem{thm}{Theorem}[section]
\newtheorem{lem}[thm]{Lemma}
\newtheorem{prop}[thm]{Proposition}
\newtheorem{cor}[thm]{Corollary}

\theoremstyle{definition}

\newtheorem{defn}[thm]{Definition}

\newtheorem{conj}[thm]{Conjecture}

\numberwithin{equation}{section}
\numberwithin{figure}{section}

\def\ds{\displaystyle}
\def\:{\colon}
\def\.{\cdot}

\def\<{\left\langle}
\def\>{\right\rangle}
\def\({\left(}
\def\){\right)}
\def\ph#1{\phantom{#1}}
\def\epsilon{\varepsilon}
\def\phi{\varphi}

\def\subset{\subseteq}

\def\leq{\leqslant}
\def\geq{\geqslant}

\def\bar#1{\overline{#1}}

\def\tilde#1{\widetilde{#1}}
\def\iso{\cong}

\DeclareMathOperator{\Char}{char}

\DeclareMathOperator{\im}{im}

\def\CP{\mathbb{C}\mathrm{P}}

\def\F{\mathbb{F}}

\def\k{\Bbbk}

\def\Z{\mathbb{Z}}

\DeclareMathOperator{\TAQ}{TAQ}

\DeclareMathOperator{\exc}{exc}

\def\undervee#1{\overset{#1}\vee}

\DeclareMathOperator{\Sq}{Sq}
\DeclareMathOperator{\dlQ}{Q}
\def\Einfty{$\mathcal{E}_\infty$ }

\title[Characteristics for \Einfty  ring spectra]
{Characteristics for \boldmath\Einfty ring spectra}
\author{Andrew Baker}
\date{\version}
\address{
School of Mathematics \& Statistics, University of Glasgow,
Glasgow G12 8QW, Scotland.}
\email{a.baker@maths.gla.ac.uk}
\urladdr{http://www.maths.gla.ac.uk/$\sim$ajb}
\thanks{The mathematics described in this paper is partly
based upon work supported by the National Science Foundation
under Grant No.~0932078~000 while the author was in residence
at the Mathematical Sciences Research Institute in Berkeley
California, during the Spring 2014 semester.  \\
The author would like to thank the following for helpful
comments: Peter Eccles, Paul Goerss, Rolf Hoyer, Niko
Naumann, Nigel Ray, Birgit Richter, Vesna Stojanoska,
John Rognes, Markus Szymik, Grant Walker and Zhouli Xu. \\
\emph{This paper will appear in the proceedings of the
Alpine Algebraic and Applied Topology Conference}}
\keywords{stable homotopy theory, \Einfty ring spectrum,
commutative $S$-algebra, cell algebra, power operations}
\subjclass[2010]{Primary 55P43; Secondary 55P42, 55P48}

\begin{document}

\begin{abstract}
We introduce a notion of characteristic for connective
$p$-local \Einfty  ring spectra and study some basic
properties. Apart from examples already pointed out by
Markus Szymik, we investigate some examples built from
Hopf invariant~$1$ elements in the stable homotopy
groups of spheres and make a series of conjectures about
spectra for which they may be characteristics; these
appear to involve hard questions in stable homotopy
theory.
\end{abstract}

\maketitle

\section*{Introduction}

In ordinary ring theory, the characteristic of a unital
ring is really part of the structure, although often not
introduced in elementary courses except in the context
of fields. Less standard is a generalisation of the notion
to algebras over a commutative ring and we discuss this
in Section~\ref{sec:Algebra}. The main aim of this note
is introduce an appropriate notion of characteristic for
derived commutative rings, at least in the topological
context of commutative $S$-algebras (also known as
\Einfty  ring spectra). Our approach could be extended
to other versions of derived commutative rings such as
simplicial commutative algebras or $\mathcal{E}_\infty$-algebras
over a fixed commutative ring, but we focus on the
topological version.

In~\cites{MS:Char,MS:ChromChar}, Markus Szymik introduced
a notion of characteristic for a commutative $S$-algebra.
We consider what properties a more general notion of
characteristic might be expected to have in this setting,
at least for connective algebras localised at a prime.
We do not attempt to work in the chromatic setting since
we rely on the theory of minimal atomic commutative
$S$-algebras which does not seem to extend to such an
intrinsically non-connective context.

As well as setting up a general notion of characteristic,
we discuss possible candidates for characteristics of
some important standard examples, and state conjectures
which appear to involve non-trivial questions in stable
homotopy theory. One possible approach to proving these
might involve old work of Joel Cohen~\cite{JMC:Decomp},
however, to date we have been unable to carry out such
a programme. Further related ideas are explored
in~\cite{EinftyHopfinvt1}.

We will assume the reader is familiar with the basic
theory of \Einfty  ring spectra in their avatar as
commutative $S$-algebras~\cite{EKMM}, discussions of
cellular aspects can be found in~\cites{BGRtaq,BP-Einfinity,TAQI}.

\tableofcontents

\section{Motivation: Characteristics in Algebra}
\label{sec:Algebra}

If $R$ is a (not necessarily commutative) ring with
unity $1\neq0$, there is a ring homomorphism
$\eta_R\:\Z\to R$ called the \emph{unit} or
\emph{characteristic} homomorphism, defined by
\[
\eta_R(n) = n1 =
\begin{cases}
\ph{-(}\underset{n}{\underbrace{1+\cdots+1}}&\text{if $n>0$}, \\
-(\underset{-n}{\underbrace{1+\cdots+1}})&\text{if $n<0$}, \\
\ph{1+\cdots1}0&\text{if $n=0$}.
\end{cases}
\]
Since $1\in R$ is non-zero, $\ker\eta_R$ is a proper
ideal of $\Z$ and there is a quotient monomorphism
$\bar\eta_R\:\Z/\ker\eta_R\to R$ which allows us to
identify the quotient ring $\Z/\ker\eta_R$ with image
$\eta_R\Z\subset R$, the \emph{characteristic subring}
of~$R$. Thus there is a unique non-negative integer
$\Char R\geq0$ such that $\ker\eta_R=(\Char R)\lhd\Z$,
and this is called the \emph{characteristic} of~$R$.

We can generalise this to unital $\k$-algebras over
a commutative ring~$\k$. Here $\k$-algebra is used
in its most general sense: a \emph{$\k$-algebra} $A$
is a unital ring equipped with a unital homomorphism
$\eta_A\:\k\to A$ whose image is central. The ideal
$\ker\eta_A\lhd\k$ is not necessarily principal, and
the quotient homomorphism $\bar\eta_A\:\k/\ker\eta_A\to A$
defines what might reasonably be called the
\emph{characteristic subalgebra} of~$A$. This
construction is functorial with respect to $\k$-algebra
homomorphisms, i.e., given an algebra homomorphism
$\phi\:A\to B$, there is a commutative diagram
\[
\xymatrix{
\k/\ker\eta_A\ar[r]^{\phi_0}\ar[d]_{\bar\eta_A} & \k/\ker\eta_B\ar[d]^{\bar\eta_B} \\
A\ar[r]^{\phi} & B
}
\]
and in particular, if $\phi$ is an isomorphism, so
is $\phi_0$.

The latter generalisation seems natural and we use
it as motivation for our discussion of the analogue
for commutative $S$-algebras. In that context there
do not appear to be obvious notions of ideals or
quotient objects (see~\cite{MH:Ideals} for recent
work on related questions), so care is needed in
making suitable definitions. There are some features
of a notion of characteristic in that setting which
seem desirable, in particular some kind of functoriality
and homotopy invariance. We tacitly assume that a
characteristic of a commutative $S$-algebra $R$ should
be a factorisation of its unit $\iota_R\:S\to R$ of
the form
\[
\xymatrix{
S\ar[r]_{\iota_{R_0}}\ar@/^15pt/[rr]^{\iota_{R}}
           & R_0 \ar[r]_{\bar\iota_R} & R \\
}
\]
where $\bar\iota_R$ is a morphism of commutative
$S$-algebras. Motivated by obvious functoriality
properties of characteristics for $\k$-algebras,
the following basic properties appear to be
reasonable requirements.

\medskip
\noindent
\textbf{Functoriality:}
If $f\: R\to R'$ is a morphism of commutative
$S$-algebras, then there is a diagram of commutative
$S$-algebras
\[
\xymatrix{
R_0\ar[r]^{f_0}\ar[d]_{\bar\iota_R} & R'_0\ar[d]^{\bar\iota_{R'}} \\
R\ar[r]^{f} & R'
}
\]
which is homotopy commutative.

\medskip
\noindent
\textbf{Homotopy invariance:} If $g\:R\xrightarrow{\,\sim\,}R''$
is a weak equivalence of commutative $S$-algebras then there
is a diagram of commutative $S$-algebras
 \[
\xymatrix{
R_0\ar[r]^{g_0}_\sim \ar[d]_{\bar\iota_R}
                & R''_0\ar[d]^{\bar\iota_{R''}} \\
R\ar[r]^{g}_\sim & R''
}
\]
which is homotopy commutative.

\medskip
Functoriality implies that characteristics of homotopy
equivalent commutative $S$-algebras are homotopy equivalent.

We will show that our definition of characteristic does
possess homotopy invariance, but does not appear to satisfy
functoriality in this sense, but nevertheless it does satisfy
a weaker version of this property.

\section{Background material on $S$-modules and commutative
$S$-algebras}\label{sec:Background}

We will assume the reader is familiar the framework provided
by~\cite{EKMM}, in particular we will work with the simplicial
monoidal model category of $S$-modules $\mathscr{M}_S$ and
the associated simplicial model category of commutative
$S$-algebras $\mathscr{C}_S$. Actually we will work with
the $p$-local versions of these for some prime~$p$, and
later~$S$ will denote the $p$-local sphere spectrum but
no essential differences occur in that setting. We will
write $\iota_A\:S\to A$ for the unit of a commutative
$S$-algebra, which is taken to be part of its structure.

We choose a cofibrant replacement $S^0\xrightarrow{\;\sim\;}S$
for~$S$ in model category $\mathscr{M}_S$ of $S$-modules
of~\cite{EKMM} (for example we could use the functorial
cofibrant replacement). We may consider the slice category
$S^0/\mathscr{M}_S$ of $S$-modules under~$S^0$. Every
commutative $S$-algebra $A$ admits a canonical morphism
of $S$-modules
\[
\xymatrix{
S^0\ar[r]_{\sim}\ar@/^15pt/[rr] & S\ar[r]_{\iota_A} & A \\
}
\]

\medskip
\noindent
making it an object of $S^0/\mathscr{M}_S$. This gives
rise to a functor
\[
\tilde{\mathbb{U}}\: \mathscr{C}_S \to S^0/\mathscr{M}_S
\]
which has a left adjoint
\[
\tilde{\mathbb{P}}\: S^0/\mathscr{M}_S \to \mathscr{C}_S,
\]
the reduced free functor~\cite{TAQI}*{section~5} which
gives a Quillen adjunction
\[
\xymatrix{
{\mathscr{C}_S} \ar@/_8pt/[rr]_{\tilde{\mathbb{U}}}
 && {S^0/\mathscr{M}_S} \ar@/_8pt/[ll]_{\tilde{\mathbb{P}}}
}
\]

In $S^0/\mathscr{M}_S$, coproducts are defined using
pushouts in $\mathscr{M}_S$, and we use the symbol
$\undervee{S^0}$ to indicate such coproducts. So for
$X,Y\in S^0/\mathscr{M}_S$, $X\undervee{S^0}Y$ is the
pushout of $X\leftarrow S^0\to Y$ in $\mathscr{M}_S$,
and since the reduced free algebra functor
$\tilde{\mathbb{P}}\:\mathscr{M}_S\to\mathscr{C}_S$
is a left adjoint it preserves coproducts, hence
\[
\tilde{\mathbb{P}}(X\undervee{S^0}Y) \iso
      \tilde{\mathbb{P}}X \wedge \tilde{\mathbb{P}}Y.
\]

In the $p$-local connective setting, we will use ideas
on minimal atomic $S$-modules and commutative $S$-algebras
which may be found in~\cites{Hu-Kriz-May,AJB&JPM,BGRtaq}.
In particular, the notions of nuclear CW $S$-modules and
commutative $S$-algebras will play a central r\^ole in
our work. This depends in turn on the theory of cellular
and CW objects~\cite{EKMM} in $\mathscr{M}_S$ and
$\mathscr{C}_S$. When discussing cellular constructions
we will refer to multiplicatively defined cell objects
built in $\mathscr{C}_S$ using the phrase \emph{\Einfty
cell}, and reserve \emph{cell} for objects built in
$\mathscr{M}_S$. Given a cell object $X$ in $\mathscr{M}_S$
we will write $X^{[n]}$ for the $n$-skeleton, while
for a cell object~$Y$ in $\mathscr{C}_S$ we will write
$Y^{\langle n\rangle}$. Finally, given a morphism~$f$
out of a cell object we will write $f^{[n]}$ or
$f^{\langle n\rangle}$ for the restriction to the
$n$-skeleton.

Following the Referee's suggestion, we summarise
briefly some of the key ideas about minimal atomic
$S$-modules and commutative $S$-algebras developed
in~\cites{AJB&JPM,BGRtaq}. In the following we
switch viewpoints and regard spectra as $S$-modules.

First we recall from~\cite{AJB&JPM}*{section~1}
that $p$-local connective $S$-module~$X$ with
$\pi_0(X)$ non-trivial and cyclic is \emph{minimal
atomic} if any map $f\:Y\to X$ for which
$\pi_0(f)\otimes\F_p$ and $\pi_n(f)$ ($n\geq0$)
are monomorphisms is a weak equivalence. Then~$X$
is minimal atomic if and only if the Hurewicz
homomorphism $\mathrm{h}\:\pi_*(X)\to H_*(X;\F_p)$
is trivial in all positive degrees. Notice that
this characterisation does not depend on a choice
of CW structure on~$X$; another useful way to
characterise minimal atomic spectra is in terms
of a special kind of CW structure.
Following~\cite{AJB&JPM}*{section~2}, a CW
$S$-module~$Z$ is \emph{nuclear} if $Z^{[0]}=S^0$
and for each $n\geq1$, the attaching map of
the $(n+1)$-cells
\[
f_n\:\bigvee_i S^n\to Z^{[n]}
\]
satisfies
\[
\ker\pi_n(f_n)
\subseteq p\,\pi_n\biggl(\bigvee_i S^n\biggr).
\]
Then $X$ is minimal atomic if and only if there
is a weak equivalence $f\:Z\to X$ where $Z$ is
a nuclear CW $S$-module.

Following~\cite{BGRtaq}*{section~3}, a $p$-local
connective commutative $S$-algebra $A$ with $\pi_0(A)$
non-trivial and cyclic is \emph{minimal atomic}
if given any morphism of commutative $S$-algebras
$g\:B\to A$ for which $\pi_0(g)\otimes\F_p$ and
$\pi_n(g)$ ($n\geq0$) are monomorphisms is a weak
equivalence. Then $A$ is minimal atomic if and
only if the $\TAQ$-Hurewicz homomorphism (induced
by the universal derivation $A\to\Omega_S(A)$)
$\mathrm{taq}\:\pi_*(X)\to\TAQ_*(A,S;\F_p)$ is
trivial in all positive degrees. There is also
a notion of nuclear commutative $S$-algebra
defined in terms of attaching \Einfty cells. A
CW commutative $S$-algebra $C$ is \emph{nuclear}
if $C^{\langle0\rangle}=S$ and for each $n\geq1$,
the attaching map of its $(n+1)$-dimensional
\Einfty cells is induced from a map
\[
g_n\:\bigvee_iS^n \to C^{\langle n\rangle}
\]
by passing to the induced morphism of commutative
$S$-algebras
$\mathbb{P}\biggl(\bigvee_iS^n\biggr)\to C^{\langle n\rangle}$
where
\[
\ker\pi_n(g_n) \subseteq p\,\pi_n\biggl(\bigvee_iS^n\biggr).
\]
Then $A$ is minimal atomic if and only if there
is a weak equivalence $b\:C\to X$ where $C$ is
a nuclear CW commutative $S$-algebra.

\section{Characteristics of connective $p$-local
commutative $S$-algebras}\label{sec:CharConnCommSalg,TAQI}

Let $\iota_R\:S\to R$ be such a connective commutative
$S$-algebra. We remark that if we started with~$R$ being
non-connective then we could replace it with its connective
cover in our discussion below, so we do not lose anything
by assuming connectivity. The induced ring homomorphism
$(\iota_R)_*\:\pi_0(S)\to\pi_0(R)$ could have a non-trivial
kernel, and also might not be surjective. If we focus on
$\pi_0(-)$ it might seem reasonable to define the characteristic
of $R$ to be this kernel. However, this neglects the kernel
of $(\iota_R)_*\:\pi_*(S)\to\pi_*(R)$ in positive degrees.
So another definition might be the (graded) kernel of
$(\iota_R)_*$. These definitions are closely wedded to the
algebra, and instead we propose a different approach which
makes the characteristic a commutative $S$-algebra equipped
with a morphism into~$R$.

\medskip
\noindent
\textbf{Conventions:}
For ease of notation and other simplifications, from now on
we work with connective $p$-local commutative $S$-algebras
$A$ for some fixed rational prime~$p>0$. Thus $S$ is to be
interpreted as the $p$-local sphere spectrum, and we will
assume that $\pi_0(A)$ is a cyclic $\Z_{(p)}$-module. The
use of \emph{finite-type} is always in the $p$-local context
of $p$-local cells or $\Z_{(p)}$-modules.

\medskip
Let $R$ be a connective $p$-local commutative $S$-algebra.
\begin{defn}\label{defn:Char-Conn}
A \emph{characteristic morphism} of $R$ is a morphism
of commutative $S$-algebras $j\:T\to R$ where $T$ is
a finite-type CW commutative $S$-algebra, where the
\Einfty  skeleta of~$T$ are defined inductively
using maps of the form
\[
\xymatrix{
{\ds\bigvee_i S^n} \ar[rr]\ar@/^15pt/[rrrr]^{f^n}
           && S \ar[rr]_(.45){\iota_{T^{\<n\>}}} && T^{\<n\>}
}
\]
factoring through the unit of $T^{\<n\>}$, and which satisfy
the conditions
\begin{subequations}\label{eq:Char-Conn}
\begin{align}
&\ker[f^n_*\:\pi_n(\bigvee_i S^n) \to \pi_n(T^{\<n\>})]
\subseteq p\,\pi_n(\bigvee_i S^n), \label{eq:Char-Conn1} \\
\im f^n_* &=
\im[(\iota_{T^{\<n\>}})_*\:\pi_n(S) \to \pi_n(T^{\<n\>})]
  \cap \ker[j^{\<n\>}_*\:\pi_n(T^{\<n\>}) \to \pi_n(R)].
          \label{eq:Char-Conn2}
\end{align}
\end{subequations}
The domain of any characteristic morphism is called a
\emph{characteristic} for~$R$.
\end{defn}

\noindent
\textbf{Note:} Condition~\eqref{eq:Char-Conn1} says that
the CW structure on $T$ is \emph{nuclear}, hence $T$ is
also \emph{minimal atomic}. Further properties of such
commutative $S$-algebras are discussed in~\cite{BGRtaq}*{section~3}.

Of course this definition begs the question of whether
such characteristics exist and also whether or not they
are in any sense unique. Notice also that the attaching
maps of \Einfty  cells all originate as maps into the
sphere spectrum~$S$.

\begin{lem}\label{lem:Char-ExistUnique}
Let $R$ be a connective $p$-local commutative $S$-algebra. \\
\emph{(a)} Characteristics for $R$ exist. \\
\emph{(b)} Suppose that $f\:R\to R'$ is a morphism of commutative
$S$-algebras and that $j\:T\to R$ and $j'\:T'\to R$ are characteristic
morphisms. Then there is a morphism of commutative $S$-algebras
$T\to T'$. \\
\emph{(c)} Suppose that $T_1$ and $T_2$ are two characteristics
for~$R$. Then there is a homotopy equivalence of commutative
$S$-algebras $T_1\xrightarrow{\,\simeq\,}T_2$. Therefore
characteristics are unique up to homotopy equivalence.
\end{lem}
\begin{proof}
We will make use of the notation in Definition~\ref{defn:Char-Conn}. \\
(a) We can build the skeleta of a nuclear CW commutative
$S$-algebra inductively making sure that conditions
of~\eqref{eq:Char-Conn} are satisfied. In detail, assuming
the $n$-skeleton $T^{\<n\>}$ as been constructed, consider
the epimorphism
\[
\xymatrix@!C0{
& \iota_{T^{\<n\>}}^{-1}
\biggl(\im[(\iota_{T^{\<n\>}})_*\:\pi_n(S) \to \pi_n(T^{\<n\>})]
  \cap \ker[j^{\<n\>}_*\:\pi_n(T^{\<n\>}) \to \pi_n(R)]\biggr)
  \ar@{->>}[ddd]
  \subseteq \pi_n(S) &   \\
  && \\
  && \\
& \im[(\iota_{T^{\<n\>}})_*\:\pi_n(S) \to \pi_n(T^{\<n\>})]
  \cap \ker[j^{\<n\>}_*\:\pi_n(T^{\<n\>}) \to \pi_n(R)]
  \subseteq \pi_n(T^{\<n\>}) &
}
\]
and after choosing a minimal set of generators for
the codomain, lift them to elements of the domain.
These can be used to form a suitable composition
\[
f^{n}\: \bigvee_i S^n
         \xrightarrow{\ph{\;\iota_{T^{\<n\>}}\;}} S
         \xrightarrow{\;\iota_{T^{\<n\>}}\;}T^{\<n\>}
\]
satisfying~\eqref{eq:Char-Conn}. The $(n+1)$-skeleton
$T^{\<n+1\>}$ is defined by the following pushout
diagram in $\mathscr{C}_S$,
\[
\xymatrix{
\mathbb{P}(\bigvee_i S^n)
 \ar@{}[dr]|{\PO} \ar[r]\ar[d]_{\tilde{f^n}}
          & \mathbb{P}(\bigvee_i D^{n+1})\ar[d] \\
T^{\langle n\rangle}\ar[r] & T^{\langle n+1\rangle}
}
\]
where $\tilde{f^n}$ is induced using the freeness
of the functor $\mathbb{P}=\mathbb{P}_S\:\mathscr{M}_S\to\mathscr{C}_S$.
The existing morphism $j^{\langle n\rangle}\:T^{\langle n\rangle}\to R$
extends to a morphism
$j^{\langle n+1\rangle}\:T^{\langle n+1\rangle}\to R$. \\
(b) We will inductively build compatible morphisms of
commutative $S$-algebras $g^n\:T^{\langle n\rangle}\to T'$.

Assume that for some $n\geq0$, we have a morphism of
commutative $S$-algebras $g^n\:T^{\langle n\rangle}\to T'$
making the following diagram of solid arrows commute, where
$\iota$ always denotes a suitable unit.
\begin{equation}\label{eq:Char-ExistUnique-b}
\xymatrix{
&& T^{\langle n\rangle}\ar@{.>}[dd]_(.5){g^n}\ar[rr]^{j^{\langle n\rangle}}
      &&R\ar[dd]^f \\
{\ds\bigvee_iS^n}\ar@{-->}@/^15pt/[rru]^{f^n}\ar@{-->}[r]
&S \ar[ru]^{\iota}\ar[rd]_{\iota} &&& \\
&&   T'\ar[rr]_{j'} && R'
}
\end{equation}
Note that we are not asserting the the right hand square
with dotted edge commutes, however the adjacent triangle
does. The dashed arrows represent the factorisation of
the attaching map $f^n$ of the \Einfty  $(n+1)$-cells
of~$T$ and the diagram of solid and dashed arrows commutes.
By construction, $j^{\langle n\rangle}\circ f^n$ is null
homotopic, hence so is the lower composition
\[
\bigvee_iS^n \xrightarrow{\ph{\;\iota_{T'}\;}} S
  \xrightarrow{\;\iota\;} T' \xrightarrow{\;j'\;} R'.
\]
Therefore the image of the induced group homomorphism
\[
\pi_n(\bigvee_iS^n) \to \pi_n(T')
\]
is contained in $\ker[(j')_*\:\pi_n(T')\to\pi_n(R)]$.
It follows that there is an extension of $g^n$ to a
morphism of commutative $S$-algebras
$g^{n+1}\:T^{\langle n+1\rangle}\to T'$. By induction
on~$n$ and passing to the colimit, we obtain a morphism
$g\:T\to T'$. \\
(c) We make use of the fact that nuclear complexes are
(minimal) atomic; see~\cite{AJB&JPM}*{proposition~2.3}
and~\cite{BGRtaq}*{theorem~3.4} for the multiplicative
case. Using this, it is enough to construct morphisms
of commutative $S$-algebras
\[
\xymatrix{
T_1 \ar@/^10pt/[rr]^{g_1} && T_2 \ar@/^10pt/[ll]^{g_2}
}
\]
by applying~(b) to the identity morphism $R\to R$. Since
the compositions $g_2\circ g_1$ and $g_1\circ g_2$ are
weak equivalences and therefore homotopy equivalences
by Whitehead's Theorem, therefore so is each $g_r$.
%
%
\end{proof}

\begin{lem}\label{lem:Char-HtpyInvce}
Let $g\:R\to R'$ be a weak equivalence of connective
$p$-local commutative $S$-algebras and let $k\:T\to R'$
be a characteristic morphism. Then there is a morphism
of commutative $S$-algebras $j\:T\to R$ such that
$g\circ j\simeq k$ and $j$ is a characteristic morphism
for~$R$.
\end{lem}
\begin{proof}
We can make use of the simplicial structure of $\mathscr{C}_S$
to form cylinder objects; for an algebra $A\in\mathscr{C}_S$
the cylinder $A\otimes I$ is the domain for homotopies
between morphisms. We also require a version of Peter May's
HELP in $\mathscr{C}_S$, see~\cite{EKMM}: Given a commutative
diagram of solid arrows
\[
\xymatrix{
C^{\langle n\rangle}\ar[dd]\ar[rr]^{i_1}\ar[dr]
  &&C^{\langle n\rangle}\otimes I\ar[dd]|\hole \ar[dr]
  &&\ar[ll]_{i_0}C^{\langle n\rangle}\ar[dd] \\
&A\ar[rr]_(.4){\sim}&&B& \\
C^{\langle n+1\rangle}\ar[rr]^{i_1}\ar@{.>}[ur]
  &&C^{\langle n+1\rangle}\otimes I\ar@{.>}[ur]
  &&\ar[ll]_{i_0}C^{\langle n+1\rangle}\ar[ul]
}
\]
there is an extension to the larger commutative diagram
with dotted arrows. Here $C$ is a connective CW algebra
and $i_0,i_1\:X\to X\otimes I$ are the two morphisms
corresponding to the ends of the cylinder on~$X$; the
vertical morphisms involve inclusions of skeleta.

We apply HELP to give the inductive step in constructing
a morphism $T\to R$. Assume that we have a suitable morphism
$j^{\langle n\rangle}\:T^{\langle n\rangle}\to R$ so that
$g\circ j^{\langle n\rangle}\simeq k^{\langle n\rangle}$.
Given a homotopy $h_n\:T^{\langle n\rangle}\otimes I\to R'$
with $h_n\circ i_0=k^{\langle n\rangle}$ and
$h_n\circ i_1=g\circ j^{\langle n\rangle}$, there is a
commutative diagram of solid arrows
\[
\xymatrix{
T^{\langle n\rangle}\ar[dd]\ar[rr]^{i_1}\ar[dr]^{j^{\langle n\rangle}}
  &&T^{\langle n\rangle}\otimes I\ar[dd]|\hole \ar[dr]^{h_n}
  &&\ar[ll]_{i_0}T^{\langle n\rangle}\ar[dd]\ar[dl]_{k^{\langle n\rangle}} \\
&R\ar[rr]_(.4){\sim}^(.4)g&&R'& \\
T^{\langle n+1\rangle}\ar[rr]^{i_1}\ar@{.>}[ur]_(.65){j^{\langle n+1\rangle}}
  &&T^{\langle n+1\rangle}\otimes I\ar@{.>}[ur]_(.65){h_{n+1}}
  &&\ar[ll]_{i_0}T^{\langle n+1\rangle}\ar[ul]_{k^{\langle n+1\rangle}}
}
\]
and by HELP an extension to a larger diagram exists. The
induction is grounded in the case $n=0$ where
$T^{\langle0\rangle}=S$.
\end{proof}

Here is a summary of our results which provide substitutes
for the functoriality and homotopy invariance conditions
of Section~\ref{sec:Algebra}.

\begin{thm}\label{thm:Char}
Let $\mathscr{C}_S$ be the category of $p$-local
commutative $S$-algebras. \\
\emph{(a)} Every connective object $R\in\mathscr{C}_S$
has a characteristic $\Char R$ which is well-defined
up to homotopy equivalence in $\mathscr{C}_S$. \\
\emph{(b)} Given a morphism of commutative $S$-algebras
$R\to R'$, there is a morphism $\Char R\to\Char R'$. \\
\emph{(c)} Given a weak equivalence $g\:R\to R'$, there
are characteristic morphisms $j\:T\to R$ and $k\:T\to R'$
which fit into the following homotopy commutative diagram.
\[
\xymatrix{
& \ar[dl]_j T \ar@{}[d]|(.6)\simeq \ar[dr]^k & \\
R\ar[rr]_
g & & R'
}
\]
\end{thm}

The next result shows that our notion of characteristic
really only depends on the kernel of the induced algebraic
unit; in particular all commutative $S$-algebras with
torsion free homotopy have the same characteristics.
\begin{prop}\label{prop:char-ker=}
Suppose that $\iota_R\:S\to R$ and $\iota_{R'}\:S\to R'$
are two commutative $S$-algebras so that
$\ker(\iota_R)_*\subseteq\ker(\iota_{R'})_*\subseteq\pi_*(S)$.
Then there is a morphism of commutative $S$-algebras
$\Char R\to\Char R'$.
\end{prop}
\begin{proof}
It is straightforward to show that for $T=\Char R$
and $T'=\Char R'$, given a morphism
$T^{\langle n\rangle}\to T'$ on the $n$-skeleton,
there is an extension to the $(n+1)$-skeleton,
giving a commutative diagram
\[
\xymatrix{
& \ar[dl]S\ar[dr] & \\
T^{\langle n\rangle}\ar[dr]\ar[rr]
      && T^{\langle n+1\rangle}\ar[dl] \\
     &T'&
}
\]
and by induction on $n$, it follows that a morphism
$T\to T'$ exists.
\end{proof}
\begin{cor}\label{cor:char-ker=}
Suppose that $\iota_R\:S\to R$ and $\iota_{R'}\:S\to R'$
are two commutative $S$-algebras for which
$\ker(\iota_R)_*=\ker(\iota_{R'})_*\subseteq\pi_*(S)$.
Then there is a homotopy equivalence of commutative
$S$-algebras $\Char R\xrightarrow{\;\simeq\;}\Char R'$.
\end{cor}
\begin{proof}
By the Proposition there are morphisms $\Char R\to\Char R'\to\Char R$
whose compositions are weak equivalences since $\Char R$
and $\Char R'$ are both (minimal) atomic.
\end{proof}

We end this discussion with an observation that relates
the notion of a characteristic to that of a nuclear
$S$-module. Let $R$ be a connective $p$-local commutative
$S$-algebra. Then as described in~\cite{AJB&JPM}, beginning
with the $0$-cell $S^0$ we can construct a nuclear CW
complex~$X$ by attaching cells only to the bottom cell,
and a map $X\to R$ which induces a monomorphism on
$\pi_*(-)$ and $\F_p\otimes\pi_0(-)$ (i.e., a core for~$R$).
There is a unique extension to a morphism of commutative
$S$-algebras $\tilde{\mathbb{P}}X\to R$.
\begin{prop}\label{prop:Char-Core}
If\/ $\tilde{\mathbb{P}}X$ is a minimal atomic commutative
$S$-algebra then the morphism $\tilde{\mathbb{P}}X\to R$
described above is a characteristic for~$R$.
\end{prop}
\begin{proof}
Suppose that $R_0\to R$ is a characteristic morphism
for~$R$. Then it is straightforward to construct a
map $X\to R_0$ under $S^0$, and this has a unique
extension to a morphism of commutative $S$-algebras
$\tilde{\mathbb{P}}X\to R_0$. Another cellular
argument constructs a morphism of commutative $S$-algebras
$R_0\to\tilde{\mathbb{P}}X$. If $\tilde{\mathbb{P}}X$
is minimal atomic then the two composite endomorphisms
are weak equivalences and so each morphism is also a
weak equivalence.
\end{proof}
The reason for the conditional statement here is that
in general for a CW complex $S^0\to Y$, while
$\tilde{\mathbb{P}}Y$ being minimal atomic implies~$Y$
is minimal atomic, the converse need not be true. An
example for $p=2$ is provided by
$Y=\Sigma^{-2}\Sigma^\infty\CP^\infty$,
see~\cite{AB&BR:Arolla2012}. In the  examples we will
consider later, this condition will in fact be satisfied.

\section{Examples}\label{sec:Examples}

\subsection{Prime power characteristics}\label{subsec:Primepower}
We begin with a generalisation of examples discussed
in~\cite{MS:Char}. For a prime $p$ and $r\geq1$, we may
form $S/\!/p^r$ as the pushout of the solid square in
\[
\xymatrix{
& \ar[dl]_{\widetilde{p^r}} \mathbb{P}S^0
       \ar@{ >->}[r]\ar@{ >->}[d]\ar@{}[dr]|{\PO}
  & \mathbb{P}D^1\ar@{ >->}[d] \\
S & \ar@{->>}[l]^(.4){\sim}\tilde{S} \ar@{ >->}[r] & S/\!/p^r
}
\]
where $\widetilde{p^r}$ is the multiplicative extension
of a map $S^0\to S$ of degree~$p^r$ and the left-hand
triangle is the functorial cofibration/acyclic fibration
factorisation of $\widetilde{p^r}$. There is also a homotopy
equivalence of commutative $S$-algebras
\[
\widetilde{\mathbb{P}}C_{p^r} \xrightarrow{\;\simeq\;} S/\!/p^r
\]
where $C_{p^r}$ is the mapping cone of a map $S^0\to S^0$
of degree~$p^r$ viewed as an object of $S^0/\mathscr{M}_S$.
Notice that $\pi_0(S/\!/p^r)=\Z/p^r$, so this ring has
characteristic~$p^r$. Furthermore, there is a (non-unique)
morphism of commutative $S$-algebras $S/\!/p^r\to H\Z/p^r$.

When $r=1$, Steinberger's splitting result~\cite{LNM1176}*{theorem~III.4.1}
(see also~\cite{Nishida}*{section~10} for a more recent
approach) implies that the spectrum $S/\!/p^r$ splits
as a wedge of suspensions of the Eilenberg-Mac~Lane
spectrum~$H\F_p=H\Z/p$. This means that the unit
induces a ring homomorphism $\pi_*(S)\to\pi_*(S/\!/p)$
whose kernel contains~$p$ and all positive degree
elements of the domain. This shows that $S/\!/p$
is a characteristic of $H\F_p$. More generally,
we have the following.
\begin{lem}\label{lem:prime-char}
Let $R$ be a $p$-local commutative $S$-algebra
for which the ring $\pi_0(R)$ has characteristic~$p$.
Then $S/\!/p$ is a characteristic of~$R$ and $R$
is a wedge of suspensions of~$H\F_p$.
\end{lem}
\begin{proof}
The unit map $S\to R$ factors through the mapping
cone $C_p$ so there is an associated morphism of
commutative $S$-algebras $S/\!/p\to R$ which also
factors the unit. By the above discussion, the
kernel of $\pi_*(S/\!/p)\to\pi_*(R)$ contains~$p$
and all positive degree elements. Furthermore,
$\pi_*(R)$ is a graded $\F_p$-vector space and
a standard argument shows that given a basis
$b_i$ ($i\in I)$ there is a weak equivalence
\[
R \sim \bigvee_{i\in I}\Sigma^{|b_i|}H\F_p.
\qedhere
\]
\end{proof}

When $r>1$, we still have $\pi_0(S/\!/p^r)=\Z/p^r$,
but the splitting result of~\cite{LNM1176}*{theorem~III.4.2}
requires further conditions and does not imply
a splitting of $S/\!/p^r$. When~$p$ is odd,
the element $\alpha_1\in\pi_{2p-3}(S)$ survives
in $\pi_{2p-3}(S/\!/p^r)$, and we can attach
an \Einfty cell to kill its image, giving a
commutative $S$-algebra $S/\!/(p^r,\alpha_1)$
for which
\[
\pi_{2p-2}(S/\!/(p^r,\alpha_1)) = \Z_{(p)}u'_1,
\]
where the new cell gives an integral homology
class $x'_{2p-2}\in H_{2p-2}(S/\!/(p^r,\alpha_1))$
so that the Hurewicz image of $u'_1$ is $px'_{2p-2}$.
In order to obtain a commutative
$S$-algebra satisfying the condition that~$1$ is
in the image of $\beta\mathcal{P}^1_*$ acting on
$H_{2p-1}(-)$, we need to attach an \Einfty cell
of dimension $2p-1$ to kill~$u'_1$, giving
$S/\!/(p^r,\alpha_1,u'_1)$ which does split as
a wedge of suspensions of the Eilenberg-Mac~Lane
spectra $H\Z/p^s$ for $0\leq s\leq r$. It is
tempting to state the following.
\begin{conj}\label{conj:oddprime^r}
For an odd prime $p$ and $r>1$,\/ $S/\!/(p^r,\alpha_1)$
is a characteristic for $H\Z/p^r$.
\end{conj}

When $p=2$, we have a similar situation with
$\eta$ in place of $\alpha_1$. Then
$1\in H_0(S/\!/(2^r,\eta,u'_1))$ is in the
image of $\Sq^3$ acting on $H_{3}(S/\!/(2^r,\eta,u'_1))$,
so $S/\!/(2^r,\eta,u'_1)$ splits as a wedge of
suspensions of the Eilenberg-Mac~Lane spectra
$H\Z/2^s$ for $0\leq s\leq r$. Again we are
led to make a conjecture.
\begin{conj}\label{conj:2^r}
For $r>1$,\/ $S/\!/(2^r,\eta)$ is a characteristic
for $H\Z/2^r$.
\end{conj}

The reader is warned that we have no direct evidence for
this and the discussion of Section~\ref{subsec:HopfInvt-2}
suggests that it may be far too optimistic.

\subsection{Odd-primary examples associated with Hopf
invariant one elements}\label{subsec:HopfInvt-odd}

Let $p$ be an odd prime. Then we can kill the element
$\alpha_1$ to form $S/\!/\alpha_1$ which has an \Einfty
morphism $S/\!/\alpha_1\to\ell$ to the connective cover
of the Adams summand of $KU_{(p)}$ which is known to
possess an essentially unique \Einfty  structure by
results of~\cites{AB&BR:HGamma,AB&BR:ConnCovers}. Then
\begin{align*}
\pi_{2p-2}(S/\!/\alpha_1) &= \Z_{(p)}u_1, \\
H_{2p-2}(S/\!/\alpha_1;\Z_{(p)}) &= \Z_{(p)}x_{2p-2},
\end{align*}
where the Hurewicz image of $u_1$ is $px_{2p-2}$.

As in the discussion for $S/\!/p^r,\alpha_1,u'_1$, we
can apply Steiberger's splitting results to show that
$S/\!/\alpha_1,u_1$ splits as a wedge of $H\Z_{(p)}$
and suspensions of Eilenberg-Mac~Lane spectra $H\Z_{p^r}$
for various $r\geq1$.
\begin{conj}\label{conj:oddprime}
For an odd prime $p$ and $r>1$,\/ $S/\!/\alpha_1$ is
a characteristic for $H\Z_{(p)}$.
\end{conj}

\subsection{$2$-primary examples associated with Hopf
invariant one elements}\label{subsec:HopfInvt-2}

In this section we take $p=2$ and consider the four
elements $2\in\pi_0(S)$, $\eta\in\pi_1(S)$, $\nu\in\pi_3(S)$
and $\sigma\in\pi_3(S)$ of Hopf invariant~$1$. We also
set $H=H\F_2$ and denote the mod~$2$ dual Steenrod
algebra by $\mathcal{A}_*=\mathcal{A}(2)_*$ and the
Steenrod algebra by $\mathcal{A}^*=\mathcal{A}(2)^*$.
By results of~\cite{BP-Einfinity}, the mod~$2$ homology
of $S/\!/\eta$, $S/\!/\nu$ and $S/\!/\sigma$ are all
polynomial on admissible Dyer-Lashof monomials on
generators $x_1\in H_1(S/\!/2)$, $x_2\in H_2(S/\!/\eta)$,
$x_4\in H_4(S/\!/\nu)$ and $x_8\in H_4(S/\!/\sigma)$:
\begin{align*}
H_*(S/\!/2) &= \F_2[\dlQ^Ix_1:\exc(I)>1],
& H_*(S/\!/\eta) &= \F_2[\dlQ^Ix_2:\exc(I)>2], \\
H_*(S/\!/\nu) &= \F_2[\dlQ^Ix_4:\exc(I)>4],
& H_*(S/\!/\sigma) &= \F_2[\dlQ^Ix_8:\exc(I)>8].
\end{align*}
The $\mathcal{A}_*$-coaction on the generator $x_{2^d}$
is given by
\[
\psi(x_{2^d}) = \zeta_1^{2^d}\otimes1 + 1\otimes x_{2^d},
\]
and the coaction on a generator $\dlQ^Ix_2$ can be found
using formulae in~\cite{Nishida}, at least in principal.
Dually, the Steenrod action satisfies
\[
\Sq^{2^d}_*(x_{2^d}) = 1,
\]
and we also have
\[
\Sq^{2^{d+k}}_*(x_{2^d}^{2^{k}}) = 1.
\]
This suggests that $\nu$ and $\sigma$ might map to zero
in $\pi_*(S/\!/\eta)$, at least modulo Adams filtration~$2$.
We will examine this in detail later.

There are \Einfty  morphisms
\[
\xymatrix{
S/\!/2\ar[drrr] && S/\!/\eta\ar[dr]
&& S/\!/\nu\ar[dl]  && S/\!/\sigma\ar[dlll] \\
&& & H & &&
}
\]
which induce ring homomorphisms
\[
\xymatrix{
H_*(S/\!/2)\ar[drrr] && H_*(S/\!/\eta)\ar[dr]
&& H_*(S/\!/\nu)\ar[dl]  && H_*(S/\!/\sigma)\ar[dlll] \\
&& & H_*(H) & && \\
&&& \mathcal{A}_*\ar@{=}[u] &&&
}
\]
which allow us to identify their images as
$\mathcal{A}_*$-subcomodule algebras of the dual
Steenrod algebra $\mathcal{A}_*$.
\begin{lem}\label{lem:Images}
The above homomorphisms give epimorphisms of
$\mathcal{A}_*$-comodule algebras
\[
\xymatrix{
H_*(S/\!/2)\ar[r]^(.45){\iso} & \F_2[\zeta_s : s\geq1]
& H_*(S/\!/\eta)\ar@{->>}[r]& \F_2[\zeta_s^2 : s\geq1] \\
H_*(S/\!/\nu)\ar@{->>}[r] & \F_2[\zeta_s^4 : s\geq1]
& H_*(S/\!/\sigma)\ar@{->>}[r]& \F_2[\zeta_s^8 : s\geq1]
}
\]
onto subalgebras of $\mathcal{A}_*$.
\end{lem}
\begin{proof}
It is easy to see that in each case, the generator
$x_{2^d}$ maps to $\zeta_1^{2^d}$. It follows that
\[
\dlQ^{(i_1,\ldots,i_\ell)}x_{2^d} \longmapsto
\begin{cases}
 (\dlQ^{(i_1/2^d,\ldots,i_\ell/2^d)}\zeta_1)^{2^d}
   & \text{\textrm{if every $i_r$ is divisible by $2^d$}}, \\
 \ph{abcdefghk} 0 & \text{\textrm{otherwise}},
\end{cases}
\]
hence the image is a subring of $\F_2[\zeta_s^{2^d}:s\geq1]$.
Making use of Steinberger's determination of the
Dyer-Lashof action on $\mathcal{A}_*$ we see that the
image contains all of the elements $\zeta_s^{2^d}$,
therefore the image is exactly this subring of $2^d$-th
powers.
\end{proof}

There is a commutative diagram of \Einfty  morphisms
as indicated by solid arrows
\begin{equation}\label{eq:Morphisms?}
\xymatrix{
S/\!/2 \ar[r] &MO\ar@/^15pt/[rrrdd]& &&  \\
S/\!/\eta\ar[r]\ar@{.>}[u] & MU\ar[r]\ar[u] & kU\ar[dr] & & \\
&&& H\Z\ar[r] & H\F_2 \\
S/\!/\nu\ar[r]\ar@{.>}[uu] & MSp\ar[r]\ar[uu] & kO\ar[ur]\ar[uu] & &  \\
S/\!/\sigma\ar[r]\ar@{.>}[u] & MO\langle8\rangle\ar[r] & tmf\ar@/_15pt/[uur] &  &
}
\end{equation}
where the left-most horizontal morphisms exist because
the bottom cells of the Thom spectra support non-trivial
actions of Steenrod operations of the form $\Sq^{2^d}$
by the Wu formulae. We will consider the possible
existence of suitable morphisms corresponding to
the vertical dotted arrows.

Steinberger's work shows that there is a splitting
of $S/\!/2$ into a wedge of suspensions of $H=H\F_2$,
hence the ring homomorphism $\pi_*(S)\to\pi_*(S/\!/2)$
induced by the unit is trivial in positive degrees.
It follows that there are \Einfty  morphisms to
$S/\!/2$ from each of $S/\!/\eta$, $S/\!/\nu$ and
$S/\!/\sigma$. In each case, under the induced
homomorphism in homology, $x_{2^d}\mapsto\zeta_1^{2^d}$.
For the case of $S/\!/\eta$ we can also deduce that
such a map exists using the fact that the inclusion
of the bottom cell induces an isomorphism
\[
\pi_1(S^0) \xrightarrow{\;\iso\;}\pi_1(C_2),
\]
so in $H_*(S/\!/2)$ the $2$-cell is attached to the
bottom cell by~$\eta$ since $\Sq^2(x_1^2)=1$.

\begin{lem}\label{lem:Ceta-nu}
Under the ring homomorphism $\pi_*(S)\to\pi_*(S/\!/\eta)$
induced by the unit $S\to S/\!/\eta$, $\nu\mapsto 0$,
hence there is an \Einfty  morphism $S/\!/\nu\to S/\!/\eta$
inducing a ring homomorphism $H_*(S/\!/\nu)\to H_*(S/\!/\eta)$
under which $x_4\mapsto x_2^2$.
\end{lem}
\begin{proof}
Recall that there is a homotopy equivalence of \Einfty
ring spectra
\[
\tilde{\mathbb{P}}C_\eta \simeq S/\!/\eta.
\]
The long exact sequence for the homotopy of the mapping
cone $C_\eta$ has an exact portion
\[
\xymatrix{
\pi_3(S^1) \ar@{ >->}[r]^{\eta} \ar[d]_{\iso}
& \pi_3(S^0)\ar@{=}[dd] \ar[r]
& \pi_3(C_\eta) \ar[r]
& \pi_2(S^1) \ar@{->>}[r]^{\;\eta\;} \ar[d]_{\iso}
& \pi_2(S^0)\ar@{=}[dd]  \\
\pi_2(S^0) \ar@{=}[d] &&& \pi_1(S^0) \ar@{=}[d] & \\
\Z/2\,\eta^2 & \Z/8\,\nu & & \Z/2\,\eta & \Z/2\,\eta^2 \\
\eta^2\ar@{|->}[r]& \eta^3=4\nu &&\eta\ar@{|->}[r] & \eta^2
}
\]
showing that
\[
\pi_3(S^0)/\eta^3 = \pi_3(S^0)/4\nu
 \xrightarrow{\;\iso\;} \pi_3(C_\eta).
\]
In $H_*(S/\!/\eta)$ we have $\Sq^4(x_2^2)=1$. If we
realise $S/\!/\eta$ as a minimal CW $S$-module, its
$4$-skeleton has one cell in each of the dimensions
$0$, $2$ and $4$. The attaching map of the $4$-cell
to the $2$-skeleton is detected by $\Sq^4$, therefore
it must be one of the generators of $\pi_3(C_\eta)$
and so homotopic to a map which factors through $S^0$
where it agrees with $\pm\nu\bmod{4}$. Thus there is
a map $C_\nu\to S/\!/\eta$ which extends to a morphism
of \Einfty  ring spectra
\[
\tilde{\mathbb{P}}C_\nu \xrightarrow{\;\simeq\;}
   S/\!/\nu\to S/\!/\eta,
\]
inducing the stated homomorphism in homology.
\end{proof}

The situation for $S/\!/\nu$ and $S/\!/\sigma$
is more involved.
\begin{lem}\label{lem:nu-sigma}
There is no morphism of \Einfty  ring spectra
$S/\!/\sigma\to S/\!/\nu$ for which the induced
homology homomorphism sends $x_8$ to $x_4^2$.
\end{lem}
\begin{proof}
The mapping cone of $\nu$
has an associated long exact sequence
\[
\xymatrix{
\pi_7(S^3)\ar[r]^(.56){\nu}\ar[d]_{\iso}
& \pi_7(S^0)\ar@{ >->}[r] \ar@{=}[dd]
& \pi_7(C_\nu) \ar[r]
& \pi_7(S^4)\ar@{->>}[r]^{\nu} \ar[d]_{\iso}
& \pi_7(S^1)\ar[d]_{\iso}  \\
\pi_4(S^0)\ar@{=}[d]
&
&
&\pi_3(S^0)\ar@{=}[d]
&\pi_6(S^0)\ar@{=}[d] \\
0
&\Z/16\,\sigma
&
&\Z/8\,\nu
&\Z/2\,\nu^2
}
\]
where the element $2\nu\in\pi_3(S^0)$ has Adams
filtration~$2$. It follows that
\[
\pi_7(C_\nu) \iso \pi_7(S^0)\oplus 2\pi_7(S^4)
 \iso \Z/16\,\sigma \oplus \Z/4\,\widetilde{2\nu}.
\]
In $H_*(S/\!/\nu)$ we have $\Sq^8(x_4^2)=1$, but
it does not follow that the attaching map of the
$8$-cell to the $4$-skeleton $C_\nu$ can be taken
to be~$\sigma$. In fact, the attaching map can be
seen to be $\sigma + \widetilde{2\nu}$ since the
natural symmetrisation map
\[
C_\nu \wedge C_\nu \to
  E\Sigma_2\ltimes_{\Sigma_2}(C_\nu \wedge C_\nu)
\]
induces the fold map on the $4$-cells and a careful
analysis in integral homology shows that the attaching
map is as claimed (I learnt this argument from Peter
Eccles; see Figure~\ref{fig:1}). The upshot is that
there is no morphism of \Einfty  ring spectra
$S/\!/\sigma\to S/\!/\nu$ since there can be no map
$C_\sigma\to S/\!/\nu$ extending the unit.
\end{proof}

\begin{figure}[ht]
\[
\xymatrix{
  & C_\nu \wedge C_\nu\ar[rrrrr] & &&
  && E\Sigma_2\ltimes_{\Sigma_2}(C_\nu \wedge C_\nu) & \\
  && *+[o][F-]{\ph{x}}\ar@{-}[dd]^{\nu} \ar@{-}[ddll]
     && && *+[o][F-]{\ph{x}}\ar@{-}[dd]^{\sigma + \widetilde{2\nu}} & \\
  && && &&& \\
  *+[o][F-]{\ph{x}}\ar@{-}[dd]^{\nu} && *+[o][F-]{\ph{x}} &\ar@{~>}[rr]&
      && *+[o][F-]{\ph{x}}\ar@{-}[dd]^{\nu} & \\
  && && &&& \\
  *+[o][F-]{\ph{x}} && && && *+[o][F-]{\ph{x}} &
}
\]
\caption{}
\label{fig:1}
\end{figure}

It follows that the mapping cone of the composition
\[
\xymatrix{
S^3\vee S^7\ar[r]_{\nu\vee\sigma}\ar@/^18pt/[rr]^{\nu+\sigma}
& S^0\vee S^0\ar[r]_(.6){\mathrm{fold}} & S^0
}
\]
gives rise to the minimal atomic commutative
$S$-algebra
\[
S/\!/(\nu,\sigma) = \tilde{\mathbb{P}}C_{\nu+\sigma}.
\]
Then there are \Einfty morphisms
\[
S/\!/\nu \to S/\!/(\nu,\sigma) \to MSp \to kO
\]
and in fact $S/\!/(\nu,\sigma) \to MSp$ is
an $8$-equivalence in the classical sense.
By~\cite{BGRtaq}*{proposition~5.1}, all of
these are minimal atomic spectra.

For $S/\!/\sigma$, there is a morphism of
commutative $S$-algebras $S/\!/\sigma\to tmf$
and we expect this to be a characteristic.
There is a morphism $S/\!/\sigma\to H\F_2$
inducing a ring homomorphism
\[
H_*(S/\!/\sigma)\to H_*(H\F_2) = \mathcal{A}_*
\]
whose image consists of the eighth powers.

The following result is even more challenging
to verify.
\begin{lem}\label{lem:eta-sigma}
There is no morphism of \Einfty  ring spectra
$S/\!/\sigma\to S/\!/\eta$ for which the induced
homology homomorphism sends $x_8$ to $x_2^4$.
\end{lem}
\begin{proof}
The point is that the image of $\sigma$ in
$\pi_*(S/\!/\eta)$ is non-zero. This was shown
in an unpublished calculation by John Rognes
(private communication~2013). This image has
order~$4$ (instead of~$16$), and in the Adams
spectral sequence, it has filtration~$2$ rather
than~$1$.
\end{proof}
\begin{conj}\label{conj:eta-nu-sigma}
Each of the following is a characteristic, where
the maps are compositions of those mentioned
above with standard ones.
\[
\xymatrix{
& S/\!/(\eta,\sigma)\ar[dl]\ar[d]\ar[dr] & \\
 MU\ar[r] & kU\ar[r] & H\Z_{(2)}  \\
}
\]

\[
\xymatrix{
& S/\!/(\nu,\sigma)\ar[dl]\ar[dr] & \\
MSp\ar[rr] & & kO \\
}
\]

\[
S/\!/\sigma\to tmf
\]
\end{conj}

Let us consider what the statements in this
conjecture really amount to. The first is
equivalent to the unit homomorphism
$\pi_*(S)\to\pi_*(S/\!/(\eta,\sigma))$ being
trivial in positive degrees. The second
statement is equivalent to the equalities
\[
\ker[\pi_*(S)\to\pi_*(S/\!/\nu,\sigma)]
=
\ker[\pi_*(S)\to\pi_*(kO)]
=
\ker[\pi_*(S)\to\pi_*(MSp)],
\]
where the second equality was proved by
Stan Kochman~\cite{SOK:MAMS-III}*{part I},
but the first would also imply it. The
third statement is equivalent to the
equality
\[
\ker[\pi_*(S)\to\pi_*(S/\!/\sigma)]
=
\ker[\pi_*(S)\to\pi_*(tmf)].
\]

To end, we mention some further results on these spectra.
\begin{prop}\label{prop:kU-KO-epi}
The $2$-local morphisms $S/\!/\eta\to kU$ and $S/\!/\nu\to kO$
induce epimorphisms on $\pi_*(-)$.
\end{prop}
\begin{proof}
For $S/\!/\eta\to kU$, it suffices to show that it induces
an epimorphism on $\pi_2(-)$. This is clear since the Toda
bracket
$\langle2,\eta,1_{S/\!/\eta}\rangle\subseteq\pi_2(S/\!/\eta)$
is defined, where $S/\!/\eta$ is viewed as an $S$-module,
and the bracket is to be taken as a function on the set
\[
\pi_0(S)\times\pi_1(S)\times\pi_1(S/\!/\eta).
\]
By naturality, this maps to the Toda bracket
$\langle2,\eta,1_{kU}\rangle\subseteq\pi_2(kU)$ which
contains a generator of $\pi_2(kU)$ and has indeterminacy
$2\pi_2(kU)$ as pointed out in~\cite{GW:Notes}*{page~64}.

For $S/\!/\nu\to kO$, it is sufficient to show that the
generators $a,b$ of the groups $\pi_4(kO),\pi_4(kO)$
come from $\pi_4(S/\!/\nu),\pi_8(S/\!/\nu)$. We can
appeal to~\cite{GW:Notes}*{page~64} (see also~\cite{AJB&JPM}*{lemma~7.3}),
where it is shown that the Toda brackets $\langle8,\nu,1_{kO}\rangle$
and $\langle8,\nu,a\rangle$ contain $a,b$ respectively.
Analogues of these brackets can be defined in $\pi_*(S/\!/\nu)$
and are preimages of the $kO$ versions. Zhouli Xu has
pointed out that the Toda brackets $\langle\eta^2,\eta,2\rangle$
and $\langle\eta,\eta^2,\eta,\eta^3\rangle$ in $\pi_*(kO)$
also contain $a,b$, furthermore they make sense in
$\pi_*(S/\!/\nu)$; this time we use the traditional
Toda brackets defined in the homotopy of a ring
spectrum.
\end{proof}

There are factorisations of these \Einfty morphisms
\[
S/\!/\eta\to T_{kU}\to kU,
\quad
S/\!/\nu\to T_{kO}\to kO,
\]
where each second factor is a characteristic
morphism. Hence these characteristic morphisms
induce epimorphisms on $\pi_*(-)$. Motivated by
these examples, we are led to make a conjecture
on a characteristic morphism for~$tmf$.

\begin{conj}\label{conj:tmf-epi}
A $2$-local characteristic morphism $T_{tmf}\to tmf$
induces an epimorphism on $\pi_*(-)$.
\end{conj}

\end{document}